\documentclass{amsart}
\usepackage{amsmath}
\usepackage{amssymb}
\usepackage{amsthm}
\usepackage{tikz}
\usetikzlibrary{matrix,arrows,positioning,automata,shapes,trees}

\usepackage{color}

\newcommand{\cC}{{\mathcal C}}
\newcommand{\bC}{{\mathbf C}}
\newcommand{\bV}{{\mathbf V}}
\newcommand{\cT}{{\mathcal T}}
\newcommand{\bD}{{\mathbf D}}
\newcommand{\cH}{{\mathcal H}}

\usepackage{url}
\usepackage{color,hyperref}
\definecolor{darkblue}{rgb}{0.0,0.0,0.3}
\hypersetup{colorlinks,breaklinks,
            linkcolor=darkblue,urlcolor=darkblue,
            anchorcolor=darkblue,citecolor=darkblue}

\pgfdeclarelayer{background layer}
\pgfsetlayers{background layer,main}

\definecolor{lgr}{rgb}{0.8,0.8,0.8}
\definecolor{bkg}{rgb}{0.8,0.8,1}

\theoremstyle{plain}
\newtheorem{theorem}{Theorem}[section]
\newtheorem{lemma}[theorem]{Lemma}
\newtheorem{fact}[theorem]{Fact}
\newtheorem{proposition}[theorem]{Proposition}
\newtheorem{corollary}[theorem]{Corollary}
\theoremstyle{definition}
\newtheorem{definition}[theorem]{Definition}
\newtheorem{example}[theorem]{Example}

\newtheorem{remark}[theorem]{Remark}

\newcommand{\rep}[1]{\overline{#1}}
\newcommand{\imgs}[1]{{\mathcal I}#1}
\newcommand{\extdimgs}[1]{{\mathcal I'}#1}
\newcommand{\tiles}[1]{{\mathcal T}(#1)}
\newcommand{\maximalchains}[1]{{\mathcal C}#1}
\DeclareMathOperator{\pos}{ pos }
\DeclareMathOperator{\enc}{ enc }
\DeclareMathOperator{\dec}{ dec }
\DeclareMathOperator{\constant}{ constant }
\newcommand{\positionedchains}[1]{{\mathcal C}^{\pos}#1}
\newcommand{\torep}[1]{m_{#1\rightarrow\rep{#1}}}
\newcommand{\fromrep}[1]{m_{\rep{#1}\rightarrow #1}}
\newcommand{\powerset}{{\mathcal P}}
\newcommand{\permutator}[1]{G_{#1}}
\newcommand{\holonomy}[1]{H_{{#1}}}
\newcommand{\depth}[1]{d({#1})}

\newcommand{\SgpDec}{\textsc{SgpDec}}
\newcommand{\GAP}{\textsc{Gap}}

\title[Holonomy Decomposition]{Computational Holonomy Decomposition of Transformation Semigroups}
\author{Attila Egri-Nagy$^{1,2}$ and Chrystopher L. Nehaniv$^1$}
\address{$^1$Centre for Computer Science \& Informatics Research,
        University of Hertfordshire, College Lane, Hatfield, Herts AL10 9AB, United Kingdom
\and
$^2$Centre for Research in Mathematics, School of Computing,
Engineering and Mathematics, Western Sydney University (Parramatta Campus), Locked Bag 1797, Penrith, NSW 2751, Australia}
\email{A.Egri-Nagy@uws.edu.au,\ C.L.Nehaniv@herts.ac.uk}

\begin{document}
\maketitle
\begin{abstract}
We present an understandable, efficient, and streamlined proof of the Holonomy Decomposition for finite transformation
semigroups and automata. 
This constructive proof closely follows the existing computational
implementation.
Its novelty lies in the strict separation of several different ideas
appearing in the holonomy method.
The steps of the proof and the constructions are illustrated with
computed examples.
\end{abstract}
\tableofcontents

\section{Introduction}

One of the fundamental concepts of science and computation is the notion of \emph{change}:
a system goes from a state to another state due to external
manipulations or due to internal processes at various time-scales.
If the set of states is a continuum then we study continuous functions
and thus we do analysis. 
If we
have a set of discrete states then we do algebraic automata theory. 
A \emph{transformation semigroup} $(X,S)$ captures the concept of
change in a rigorous and discrete way. It consists of a set of
\emph{states} $X$ (analogous to \emph{phase space}),  and a set $S$ of
transformations  of the state set, $s:X\rightarrow X$ acting by
$x\mapsto x\cdot s$,
that is closed under the associative operation of function
composition. Writing $s_1s_2 \in S$ for the composite function $s_1\in S$ followed by $s_2 \in S$, we
have $x\cdot (s_1s_2)=(x\cdot s_1)\cdot s_2$, giving a (right) action of $S$ on $X$.
A fixed generating set for a transformation semigroup can be
considered as a set of input symbols, therefore automata (without
specifying initial and accepting states) and
transformation semigroups are essentially the same concepts. 

Another fundamental technique of the scientific method is \emph{decomposition}.
The holonomy decomposition is a method for finding the building blocks
of a transformation semigroup and compose them in a hierarchical structure.
This composite semigroup has a structure that promotes understanding
and it is capable of emulating the original
transformation semigroup.
Therefore, we say that the holonomy decomposition is a way of
understanding transformation semigroups.

Our aim here is to provide the simplest and most accessible proof
for the holonomy decomposition theorem by giving a construction  which
is `isomorphic' to its computational implementation \cite{SgpDec,SgpDec2014}.
The novelty of this proof is the strict separation of the several different ideas
that appear in the holonomy decomposition. Both separating them from
each other and from the technical details.  

\subsection{General Ideas}
There are four fundamental concepts used in the holonomy
decomposition. 
First we state them in their generality to aid intuition, then give a short summary how
they actually appear in the method.
\begin{description}
\item[Approximation] gives less information about a system in a way that the partial description  does not contradict the full description.
\item[Emulation] is a capability of one system producing the same
  dynamics as another one, not necessarily containing an exact copy.
\item[Compression] for repeated patterns stores the pattern once and record its occurrences.
\item[Hierarchy] is any system where the control information flows in
  one direction only and abstractions are natural operations.
\end{description}
In the holonomy decomposition, we study the action on chains of increasingly smaller subsets of the state set, recovering the original transformations at the level of singleton subsets (approximation).
Whenever the semigroup acts the same way on different subsets, we
consider those subsets equivalent and only store the action on the equivalence class representatives (compression).
These representative local actions are the building blocks and they
are aligned according to a partial order (hierarchy).
The chain semigroup and its encoded form, the cascade product can
compute everything the original transformation semigroup can (emulation).

\subsection{Mathematical Preliminaries, Notation}
A semigroup is a set $S$ together with an associative binary
operation $S\times S\rightarrow S$.
A semigroup is a \emph{monoid} if it contains the identity element.
Let $S^1$ denote the monoid we get by adjoining an identity to $S$ in case $S$ is not a monoid.
A \emph{transformation semigroup} $(X,S)$ is a finite nonempty set $X$ (the state set) together with a set $S$ of total transformations of $X$ closed under function composition.
The states are often denoted by a set of integers ${\bf
  n}=\{1,\ldots,n\}$, and the transformations by the list of images $[j_1, j_2, \ldots,
j_n]$, where $i\mapsto j_i$ for $i,j_i\in{\bf n}$. 
The action $x\mapsto x\cdot s$ on the points (states) $x\in X$ by  transformations $s\in S$ naturally extends to set of
points: $P\cdot s:=\{p\cdot s\mid p\in P\}$, $P\subseteq X$, $s\in
S$, and we have  $(P\cdot s_1 )\cdot s_2 = P \cdot (s_1s_2)$, for $s_1,s_2\in S$. Similarly, the action can also be extended to sets of sets of points or to tuples or sequences of points or sets of points.

The {\em wreath product} $(X,S)\wr(Y,T)$  of transformation semigroups 
is the transformation semigroup $(X\times Y,W)$ where 
$$W=\{(s,f) \mid  s\in S, f\in T^X\},$$
whose elements map  $X\times Y$ to itself as follows
$$(x,y)\cdot (s,f)=(x\cdot s, y\cdot f(x))$$ for $x\in X, y\in Y$. Here $T^X$ is the semigroup of  all functions  $f$ from $X$ to $T$ (under pointwise multiplication). Note we have written $y\cdot f(x)$ for the element $f(x)\in T$ applied to $y\in Y$.
The wreath product construction is associative on the class of
transformation semigroups (up to isomorphism) and can be iterated for
any number of components.

The size of the iterated wreath product grows rapidly by increasing
the number of components or by increasing their sizes.
Explicit computation with wreath products is impractical. 
This motivates the  definition of \emph{cascade products}: efficient constructions of substructures of wreath products, induced by explicit dependency functions~\cite{cascprod}.
Essentially, cascade products are transformation semigroups glued together by functions in a hierarchical tree.
More precisely, let $\big((X_{1},S_{1}),\ldots,(X_{n},S_{n})\big)$ be a fixed list of transformation semigroups (here $S_i$ are semigroups and $X_i$ the sets on which they act), and define {\em dependency functions} to be functions of the form
\[
d_i: X_1\times\ldots\times X_{i-1}\rightarrow S_i,\quad\text{for } i\in \{2\ldots n\}.
\]
A \emph{transformation cascade} is then defined to be an $n$-tuple of dependency functions $d=(d_1,\ldots,d_n)$, where $d_i$ is a dependency function of level $i$.  On the top level, $d_1$ is simply an element of the semigroup $S_1$.
The transformation cascade $d$ applied to  $(x_1,\ldots, x_n)$ is defined coordinatewise by  $x_i\cdot{d_i(x_1,\ldots,x_{i-1})}$, applying the results of the evaluated dependency functions, so that the cascade product can be regarded as a special transformation representation on the set $X_1\times\ldots\times X_{n}$.
The hierarchical structure allows us to conveniently distribute computation among the components, and perform abstractions and approximations of the system modelled as a cascade product.
In the permutation group case it is basically the Schreier-Sims algorithm~\cite{CGTHandbook} put into product form~\cite{cascprod}.


\subsection{Computational Tools} The constructive proof for the
holonomy decomposition described here is implemented in the \SgpDec~\cite{SgpDec,SgpDec2014}
software package for the \GAP~computer algebra system \cite{GAP4}.
For the verification of the correctness of the software package we use a
selection of transformation semigroups with interesting features and
corner cases.
We also have a shadow implementation of the algorithms based on
partitioned binary relations in the
\textsc{kigen} system \cite{kigen}. 

\subsection{Historical Notes}



In Krohn-Rhodes theory, the holonomy method for cascade decomposition was originally
developed by H. Paul Zeiger \cite{zeiger67a,zeiger68}, and subsequently improved by
S. Eilenberg \cite{eilenberg}, and later by several others \cite{automatanetworks2005,ginzburg_book68,
holcombe_textbook}. Variants \cite{KRTLocaldivisors,Maler2010}, and  generalizations of the theorem to the infinite case \cite{elston_nehaniv,general_holonomy} and to categories \cite{KRTforCategories} were also  studied.

The term `holonomy' is borrowed from differential geometry, since a roundtrip of composed bijective maps producing permutations is analogous to moving a vector via parallel transport along a smooth closed curve yielding change of the angle of the vector.  

The current proof is a prime example of the observation on the
development of mathematics, that proofs turn into definitions (see the
introduction of \cite{QBook}), as the way we define the chain
semigroup is the key argument of the previous proofs.

\section{Approximation}
For a transformation semigroup $(X,S)$ we describe ways to approximate
the states $x\in X$ by subsets of $X$, and to approximate the
transformations in $S$, the `behaviour' of the semigroup.

\subsection{Approximating states}
What is the current state of the system? We can answer this question precisely by giving a single element, or we can give partial information by specifying a set of states with the condition that  the current state is contained in the set. 
This way, any subset $P\subseteq X$ such that $x\in P$ can be considered as an \emph{approximation} of the state $x$.

For a particular transformation semigroup we do not need to consider all such elements of the power set $\powerset(X)$, we can restrict to those that are generated by the semigroup action.
\begin{definition}The set $\imgs_S(X)=\{X\cdot s\mid s\in S\}$ is the \emph{image set} of the transformation semigroup $(X,S)$.
\end{definition}
Note that in general $X$ itself and the singleton state sets are not
necessarily included, so we may need to add them to the image set.
\begin{definition}The \emph{extended image set} of the state set under
  the action of the semigroup is
  $\extdimgs_S(X)=\imgs_S(X)\cup\{X\}\cup\big\{\{x\}\mid x\in X\big\}$.
\end{definition}

When approximating, we may be interested in doing it step-by-step.
Since approximations are subsets, we can build successive approximations by nested subset chains.

\begin{definition}
A \emph{chain} $\bC$  is a subset of $\powerset(X)$ such that
$P \in {\bC}$ and $Q\in {\bC}$ implies $P \subseteq Q$ or $Q \subseteq
P$.
A chain $\bC$ is \emph{maximal} if it is not properly contained in any
other chain. We say that two chains $\bC$ and $\bD$ \emph{agree down
  to} $P$ if $P\in \bC\cap \bD$ and for all subsets $Q$ with $P\subseteq Q
\subseteq X$ we have $Q\in\bC \Leftrightarrow Q\in \bD$.
\end{definition}

\noindent
{\bf Observation.} Notice that $S$ acts on subset chains in $X$:  Since $P \subseteq Q$ implies $P\cdot s\subseteq Q\cdot s$, necessarily $\bC\cdot s$ is a chain if $\bC$ is. Moreover, $(\bC\cdot s_1)\cdot s_2 = \bC\cdot s_1s_2$.
However, the length of chains can become shorter under this action. \\

As mentioned before, for the holonomy decomposition we do not need the full power set.
However, we need the extended image set if we want to describe all necessary stages of
approximating a state by maximal chains.

\begin{definition}
Let $\maximalchains=\maximalchains(X,S)$ denote the set of all maximal chains in $\extdimgs_S(X)$.
\end{definition}

There is a surjective function $\eta:\cC\twoheadrightarrow X$
mapping each maximal chain $\bC$ to the element of its unique singleton $\{x\}\in \bC$.
We say $\bC$ is a {\em lift} of $x\in X$ if $\eta(\bC)=x$.

\subsection{Approximating Transformations}
A state $x$ is lifted as a maximal subset chain starting from $\{x\}$.
Consequently, for lifting transformations we need to construct
transformations mapping $\cC$ to itself.
However, simply acting on maximal chains, $\bC\mapsto \bC\cdot s$ is not a well-defined action on $\cC$, since $\bC\cdot s$ may not be maximal.

\begin{definition}A \emph{dominating  chain} of a chain $\bC$ is a maximal chain $\bD$ such that $\bC\subseteq \bD$.
\end{definition}  
\noindent There can be more than one dominating chain. For instance,
acting by a constant map on any chain would produce a singleton set, which can be dominated by all maximal chains containing that set.

For any fixed $s\in S$ we can define a (non-unique) mapping $\hat{s}:\cC\rightarrow
\cC$ by $\hat{s}(\bC)=\bD$, where $\bD$ is any fixed maximal chain
containing $\bC\cdot s$.   We can think of such an $\hat{s}$ as mapping the nested approximations $\bC$ of $x=\eta(\bC)$ to
nested approximations $\hat{s}(\bC)$ of $x\cdot s = \eta(\hat{s}(\bC))$.
We say $\hat{s}$ is \emph{consistent with chain structure} if $\bC$
and $\bC'$ agree down to $P$ then $\hat{s}(\bC)$ and $\hat{s}(\bC')$
agree down to $P\cdot s$.    


 One way to ensure this condition is to totally order
 $\extdimgs_S(X)$, and for example choose its least member that can be included when
 building a  dominating chain.
We observe there is always at least one way to choose $\hat{s}$ so
that it is consistent with chain structure.

\begin{lemma} \label{consistency}
If $\hat{s}_1$ and $\hat{s}_2$ mapping $\cC$ to itself are consistent with chain structure, then
so is the composite mapping $\hat{s}_1\hat{s}_2$. 
\end{lemma}
\begin{proof}
If  maximal chains $\bC$ and $\bC'$ agree down to $P \in \bC\cap \bC'$ then, since $\hat{s}_1$ is consistent with chain structure, $\hat{s}_1(\bC)$ and $\hat{s}_1(\bC')$ agree down to $P\cdot s_1$. Since $\hat{s}_2$ is consistent too, we have that
$\hat{s}_2(\hat{s}_1(\bC))$ and $\hat{s}_2(\hat{s}_1(\bC'))$ agree down to $(P\cdot s_1)\cdot s_2= P \cdot (s_1 s_2)$. 
\end{proof}

\begin{definition}[Chain semigroup]
Given a generating set  $A$, for $S$, for each
$a\in A$ we choose a consistent $\hat{a}$ and take $\hat{S}=\langle
\hat{a}\mid a\in A  \rangle$. Then we call the transformation
semigroup $(\cC, \hat{S})$ a \emph{chain semigroup}.
\end{definition}

We say $\hat{a}_1\cdots \hat{a}_k$ is a {\em lift} of $s\in S$  if  $s=a_1\cdots a_k$ for 
generators $a_1,\ldots a_k\in A$.  

By Lemma~\ref{consistency} it follows that any $\hat{s}=\hat{a}_1\cdots \hat{a}_k$ is
consistent, i.e., 

\begin{proposition}\label{consistencyProp}
All mappings in a chain semigroup $(\cC,\hat{S})$ are consistent with chain structure. 
\end{proposition}

\begin{remark}
\begin{enumerate}
\item We generally take just one lift $\hat{a}$ for each generator $a$ of $S$
to generate a chain semigroup, since one would often like $\hat{S}$ 
to be as small as possible. Different choices of lifts for the generators  can result in different sized $\hat{S}$. 
\item   Generally, there can be many different lifts $\hat{s}\in \hat{S}$ for fixed $s$ in $S$, 
since $s= a_1\ldots a_k = a_1'\cdots a_{\ell}'$ does not imply
$\hat{a}_1\ldots \hat{a}_k= \hat{a}_1'\ldots \hat{a}_{\ell}'$, although
both are lifts of $s$. 
\item
There is a unique maximal chain semigroup obtained by taking all
possible consistent $\hat{s}$ for $s\in S$, and letting $\hat{S}$ be
the semigroup they generate.     
\end{enumerate} 
\end{remark}

In a sense chain semigroup contains approximations of
$(X,S)$. The rest of the holonomy decomposition is about putting an efficient
notation (by embedding it into a wreath product) on this expanded semigroup.

\section{Emulation}

We need to show that a chain semigroup emulates the original semigroup.

\begin{lemma}\label{liftLemma}
There is a surjective morphism of transformation semigroups
$$(X,S)\twoheadleftarrow(\maximalchains, \hat{S}).$$ 
\end{lemma}
\begin{proof}
There is a semigroup homomorphism from $\hat{S}$ to $S$ determined by
$\hat{a}\mapsto a$, where we recall that $a$ is a generator of
$S$. It is not hard to see this is well-defined.  (And it follows, e.g., from Proposition~1.10 in \cite{automatanetworks2005}).
Since $\eta(\hat{a}(\bC))=x\cdot a$ for $x=\eta(\bC)$, the action is respected.
\end{proof}

In the final form of the holonomy decomposition we will use the following notion of emulation.

\subsection{Division}
One transformation semigroup {\em divides} another,  $(X,S) \mid  (Y,T)$,  if $(X,S)$ is a homomorphic image of a substructure of $(Y,T)$: precisely, there exists a subset $Z \subseteq Y$ and a subsemigroup  $U \leq T$, with  $z\cdot u \in Z$ for all $z\in Z, u \in U$, and a surjective  function $\theta_1:Z\twoheadrightarrow Y$  
and surjective homomorphism $\theta_2:U \twoheadrightarrow S$ such that $\theta_1(z \cdot u)=\theta_1(z)\cdot \theta_2(u)$ for all $z\in Z$ and $u \in U$. 

\section{Compression}

\subsection{Equivalence of Subsets}
On $\extdimgs_S(X)$ we define an equivalence relation by
$$ P\equiv_S Q \text{ iff } \exists s,t\in S^1 \text{ such that } P=Q \cdot s \text{ and }
Q=P\cdot t.$$
This is the equivalence relation of 'mutual reachability' under the action of $S$, and the equivalence
classes are the strongly connected components of $(X,S)$ acting on $\extdimgs_S(X)$.

It is immediate that $P\equiv_S Q\implies |P|=|Q|$. As we will see $S$ acts the same way on equivalent elements (see permutator and holonomy groups defined below), thus  the equivalence classes  provide the way to compress information in the decomposition.
For each equivalence class there will be only one component in the hierarchical decomposition.

\subsection{Group Actions}

For a subset $P\subseteq X$ we have the \emph{stabilizer semigroup}
$S_P=\left\{s\in S\mid P\cdot s=P\right\}$. If we restrict the action
of the stabilizer to $P$ we get the \emph{permutator group}
$\permutator{P}$. These groups are also called generalized Sch\"utzenberger groups \cite{LPRR1}.

In the holonomy decomposition we need the most coarse-grained
approximation possible so we have to take another homomorphic image of
$\permutator{P}$.
Considering the inclusion relation $(\extdimgs_S(X),\subseteq)$, we
call a  (lower) cover $P_i$ of a non-singleton subset
$P\in\imgs(X)$ a \emph{tile} denoted by $P_i\prec P$.
The set of all  tiles  of $P$ is denoted by $\tiles{P}$.
These are the maximal subsets of $P$ in $\extdimgs_S(X)$.
 Obvious properties of tiles are:
$$P=\bigcup_{i=1}^kP_i,\quad P_i\subseteq P_j\Longrightarrow P_i=P_j$$
where $P_i\in\tiles{P}$ and $k=|\tiles{P}|$. 
Important to note that tiles of a set may overlap, so one should think of roof tiles as the analogy.

The \emph{holonomy group} $\holonomy{P}$ is the permutation group $(\tiles{P}, \permutator{P})$ made faithful.

\subsection{Constructing Holonomy Groups}

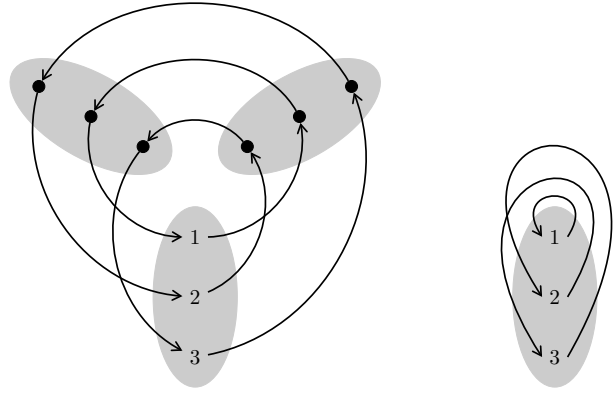
\begin{figure}
\scalebox{.8}{
\begin{tikzpicture}
\tikzstyle{blackdot}=[draw=black,circle,fill=black,inner sep=2pt]
\tikzstyle{arrow}=[thick,->,>=angle 60]
\draw node at (0,-1) (n1) {1};
\draw node at (0,-2) (n2) {2};
\draw node at (0,-3) (n3) {3};
\begin{pgfonlayer}{background layer}
\draw[lgr,fill=lgr] (0,-2) ellipse (.7cm and 1.5cm);
\end{pgfonlayer}
\begin{scope}[rotate=120]
\draw node[blackdot] at (0,-1) (np1) {};
\draw node[blackdot] at (0,-2) (np2) {};
\draw node[blackdot] at (0,-3) (np3) {};
\begin{pgfonlayer}{background layer}
\draw[lgr,fill=lgr] (0,-2) ellipse (.7cm and 1.5cm);
\end{pgfonlayer}
\end{scope}
\begin{scope}[rotate=240]
\draw node[blackdot]  at (0,-1) (npp1) {};
\draw node[blackdot]  at (0,-2) (npp2) {};
\draw node[blackdot]  at (0,-3) (npp3) {};
\begin{pgfonlayer}{background layer}
\draw[lgr,fill=lgr] (0,-2) ellipse (.7cm and 1.5cm);
\end{pgfonlayer}
\end{scope}
\draw[arrow] (n1) edge[bend right=50] (np2);
\draw[arrow] (n2) edge[bend right=50] (np1);
\draw[arrow] (n3) edge[bend right=50] (np3);
\draw[arrow] (np1) edge[bend right=50] (npp1);
\draw[arrow] (np2) edge[bend right=60] (npp2);
\draw[arrow] (np3) edge[bend right=60] (npp3);
\draw[arrow] (npp1) edge[bend right=50] (n3);
\draw[arrow] (npp2) edge[bend right=50] (n1);
\draw[arrow] (npp3) edge[bend right=50] (n2);
\end{tikzpicture}
\begin{tikzpicture}
\tikzstyle{blackdot}=[draw=black,circle,fill=black,inner sep=2pt]
\tikzstyle{arrow}=[thick,->,>=angle 60]
\draw node at (0,-1) (n1) {1};
\draw node at (0,-2) (n2) {2};
\draw node at (0,-3) (n3) {3};
\begin{pgfonlayer}{background layer}
\draw[lgr,fill=lgr] (0,-2) ellipse (.7cm and 1.5cm);
\end{pgfonlayer}
\draw[arrow] (n1.east) .. controls (.8,-.1) and (-.8,-.1).. (n1.west);
\draw[arrow] (n2.east) .. controls (2,1) and (-2.5,.5).. (n3.west);
\draw[arrow] (n3.east) .. controls (2.8,1.5) and (-2.5,1.5).. (n2.west);
\end{tikzpicture}
}
\caption{An illustrative example of how moving around in an
  equivalence class induces permutations on the elements of the
  equivalence class.}
\label{fig:roundtrips}
\end{figure}

If $P\equiv_S Q$, then there
exist mappings $m_{P\rightarrow Q}$, $m_{Q\rightarrow P}\in S$ mapping
$P$ to $Q$ bijectively ($Q$ to $P$ respectively), such that
$m_{P\rightarrow Q}m_{Q\rightarrow P}$ is the identity map restricted
to $P$ and $m_{Q\rightarrow P}m_{P\rightarrow Q}$ is the identity
restricted to $Q$ (see e.g.~\cite{LPRR1}).

It can be shown that if $P\equiv_S Q$ then
$\permutator{P}\cong\permutator{Q}$. Since there is a
bijection between $\tiles{P}$ and $\tiles{Q}$, it follows that $\holonomy{P}\cong\holonomy{Q}$. 
Moreover, `roundtrips' of mappings in the equivalence class induce
permutations on elements of the equivalence class (see schematic
drawing on Figure \ref{fig:roundtrips}). 
We can get the generators of $\permutator{R}$ by contracting
roundtrips of the form
\begin{center}
\includegraphics{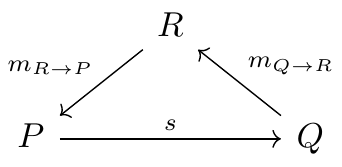}
\end{center}
where $P$ and $Q$ are elements of the equivalence class of $R$ and $s$
is a generator of $S$ mapping $P$ bijectively to $Q$.

\section{Hierarchical Structure}

The output of the holonomy decomposition algorithm is a cascade product of transformation semigroups.
So far we have established that the components of this cascade will arise from the holonomy groups of equivalence class representatives, but we still do not know how the components are put together in the cascade.

\subsection{Subduction}
The inclusion relation is naturally defined on $\extdimgs(X)$. In the subduction relation we also allow the sets to be moved by $S$.
\begin{equation*}
P\subseteq_S  Q \iff \exists s \in S^1 \text{ such that } P \subseteq Q \cdot s\quad  P,Q\in \extdimgs(X),
\label{def:subduction}
\end{equation*}
i.e. either $P\subseteq Q$ or we can transform $Q$ to include $P$
under the action of $S$. Therefore, subduction is a generalized
inclusion, i.e.\ inclusion is subduction under the action of the trivial monoid.

It is easy to  see that $\subseteq_S$ is a preorder: it is reflexive, since $P\subseteq P \cdot 1$, and it is transitive, since if  $P \subseteq Q \cdot s_1$ and $ Q \subseteq R \cdot s_2$ then $P \subseteq R \cdot s_2s_1$, thus $P\subseteq_S R$.
\begin{center}
\begin{tikzpicture}
\draw (0,0) circle (.5) node {$P$};
\draw (0,.2) circle (.8) (0,.7) node {$Q\cdot s_1$};
\draw (4,.2) circle (.8) node {$Q$};
\draw [>=latex,->]  (3.2,.2) -- (.8,.2);
\draw  (2,.4) node {$s_1$};
\draw (4,.4) circle (1.1) (4,1.2) node {$R\cdot s_2$};
\draw (8,.4) circle (1.1) (8,1.2) node {$R$};
\draw [>=latex,->]  (6.9,.4) -- (5.1,.4);
\draw  (6,.6) node {$s_2$};
\end{tikzpicture}
\end{center}

Using a common technique for preorders, we define the $\equiv_S$ equivalence relation on $\extdimgs_S(X)$ by taking subduction in both directions:
$P\equiv_S Q \iff P\subseteq_S Q \mbox{ and } Q \subseteq_S P.$

\subsection{Positioning the components: Height and Depth of Sets}

The {\em height} of a set $Q \in\extdimgs_S(X)$  is given by the function  $h:\extdimgs_S(X)\rightarrow {\mathbb N}$, which is defined by $h_S(Q)=0$ if $Q$
is a singleton, and  for $|Q|>1$, $h_S(Q)$ is defined by the length of the longest strict subduction chain(s) in the skeleton starting from a non-singleton set and ending in $Q$:
$$h_S(Q)=\max_i(Q_1 \subset_S  \cdots \subset_S Q_i=Q),$$
where $|Q_1|>1$.
The height of $(X,S)$ is $h=h_S(X)$.

It is also useful to speak of {\it depth} values, which are derived from the height values: $$d(Q)=h_S(X)-h_S(Q)+1.$$
The top level is depth 1.

Calculating the height values establishes the hierarchical levels in the decomposition, i.e.\ the number of coordinate positions in the holonomy decomposition is $h_S(X)$.

\begin{fact}[Depth never decreases]
Let $P\in \extdimgs$.
Then $\depth{P\cdot s}\geq \depth{P}$.
If $\depth{P\cdot s}= \depth{P}$ then $P\equiv_S P\cdot s$.
\end{fact}

\subsection{Positioned Chain Semigroup} By using the depth function, we can know align the members of those maximal chains on which the chain semigroup acts. 

\begin{definition}[Positioned chain]
\label{def:positionedchain}
For a maximal chain $\bC$
$$X = P_1 \supset P_2 \supset P_3 \supset \ldots \supset P_k = \{x\},$$
we take the associated  \emph{positioned chain} $\bC^{\pos}$. This is
a vector of length $h_S(X)$ where 
the slots are
empty (denoted by *) except that
$P_{i+1}$ is in position $d(P_{i})$ for $1 \leq i < k$.
For a positioned chain $\bC^{\pos}$ the content at level $i$ is $\bC^{\pos}[i]$.
\end{definition}
This puts the members of chains into coordinate slots.
By the maximality of the chain we have $\bC^{\pos}[i]\prec P_i$.
Note that a positioned chain omits $X$, since it is not a tile of anything.

We can identify the action of the chain semigroup with an
action on positioned chains denoted by $\positionedchains{(X,S)}$: 

\begin{fact}\label{chainSemiIso}
$(\cC,\hat{S}) \cong (\positionedchains,\hat{S}).$
\end{fact}
\begin{proof}
The positioned chains are in one-to-one correspondence with
the maximal chains of $\cC$ by the maps $\bC\leftrightarrow
\bC^{\pos}$, since the only missing element of the chain in the
positioned chain is $X$
itself, so it can be added without any ambiguity when recovering the
maximal chain.
\end{proof}

At each level of depth we need to know how far the approximation
proceeded so far, i.e.~we need to know what subset of the state set
are we acting on at the given depth. The value at the position is a
tile, and tiles can belong to more than one set, so we need to look
back to the first concrete value above.

\begin{definition}[state of approximation] 
$$\alpha_i(\bC^{\pos})=
\begin{cases}
X & \text{if } i =1\\
\bC^{\pos}[j] & \text{ otherwise, where }j=\max_j\left\{\bC^{\pos}[j]\neq * \text{ and } j<i\right\}, 
\end{cases}$$
$1\leq i \leq h_S(X)$.
\end{definition}

Since $\alpha_i$ only depends on $\bC^{\pos}[j]$ where $j<i$,
$\alpha_i$ is well-defined on prefixes of $\bC$ of length at least
$i-1$.
Moreover, we define $\alpha(\bC^{\pos})=\left(
  \alpha_1(\bC^{\pos}),\ldots, \alpha_{h_S(X)}(\bC^{\pos}) \right)$. 

\begin{lemma}
For all maximal chains $\bC$ and $1\leq i \leq h_S(X)$, $$\bC^{\pos}[i]\in \tiles{\alpha_i(\bC^{\pos})}$$
when $\bC^{\pos}[i]\neq *$.
\end{lemma}
\begin{proof}
This is immediate from the definition of positioned chains (Def.~\ref{def:positionedchain}). 
\end{proof}

\begin{lemma}
For all maximal chains $\bC$ it always holds the $\depth{\alpha_i(\bC^{\pos})}\geq i$. 
\end{lemma}
\begin{proof}
If $i=1$ then $\alpha_i(\bC^{\pos})=X$ and $\depth{X}=1$ so the
statement holds.
If $i>1$ then assume the statement holds by induction hypothesis for
$i$ so we show that it follows for $i+1$:

\begin{description}
\item[Case 1] $\depth{\alpha_i(\bC^{\pos})} =  i$ then by the definition of
positioned chains $\bC^{\pos}[i]$ is a tile of $\alpha_i(\bC^{\pos})$,
so at level $i+1$ the value of $\alpha_{i+1}$ will be this tile, which
is of depth at least $i+1$.
\item[Case 2] $\depth{\alpha_i(\bC^{\pos})} > i$ then
$\alpha_{i+1}(\bC^{\pos})=\alpha_i(\bC^{\pos})$ but still
$\depth{\alpha_{i+1}(\bC^{\pos})}\geq i+1$.
\end{description}
 \end{proof}

When lifting a transformation $s$, we only need to act when we are on the right level, i.e.~$\depth{\alpha_i(\hat{s}(\bC^{\pos}))}=i$.
The next lemma shows that the action of a lifted transformation respects approximation.

\begin{lemma}\label{chainaction}
For a transformation $s\in S$ and a maximal chain $\bC$, we have for all coordinate levels $i$ 
$$ (\alpha_i(\bC^{\pos}))\cdot s \subseteq \alpha_i\left((\hat{s}(\bC))^{\pos}\right).$$ 
\end{lemma}
\begin{proof} 
  Let $P_i= \alpha_i(\bC^{\pos})$ and $Q_i=
  \alpha_i((\hat{s}(\bC))^{\pos})$.
$P_1=Q_1=X$, so the statement is true for $i=1$.

By induction hypothesis, the statement holds for levels down to and
including $i$.
Trivially, $P_{i+1}\subseteq P_i$ and $Q_{i+1}\subseteq Q_i$. 
We show that $P_{i+1}\cdot s\subseteq Q_{i+1}$.

If $Q_{i+1}=Q_i$ then the statement holds since $P_{i+1}\cdot
s\subseteq P_i\cdot s\subseteq Q_i$.
Otherwise, $Q_{i+1}\subset Q_i$ and by the maximality of the chain $Q_{i+1}\in\tiles{Q_i}$.
\begin{description}
\item[Case 1] If $P_{i}\cdot s= Q_{i}$, then $\depth{P_i}\leq \depth{Q_i}$
as $P_i$ cannot be deeper than $Q_i$ 
and $\depth{Q_i}= i$ since we are on the right level and
$\depth{P_i}\geq i$ always holds.
Thus we have $P_i\equiv_S Q_i$ and $\depth{P_i}=\depth{Q_i}=i$.
Also, $\depth{P_{i+1}}\geq i+1$, therefore $P_i\neq P_{i+1}$.
Finally, $P_{i+1}\cdot s \in \hat{s}(\bC^{\pos})\Rightarrow
P_{i+1}\cdot s\subseteq Q_{i+1}$.

\item[Case 2] If $P_{i}\cdot s\subset Q_{i}$ since $\hat{s}(\bC)$ is a maximal
chain containing  $P_{i}\cdot s$ and $Q_{i+1}$ is a tile of $Q_i$, so
$P_{i}\cdot s\subseteq Q_{i+1}$, whence $P_{i+1}\cdot s\subseteq Q_{i+1}$. 
\end{description}
\end{proof}

\subsection{Holonomy Cascade Semigroup}

We build a cascade product of the holonomy groups of $(X,S)$. 
First the components.
Let $R_1,\ldots, R_k$ be the representative sets of depth $i$. Then
the $i$th component of the cascade product is defined as the transformation semigroup 
$$\cH_i=(\cT_i, \overline{H_i})=\left( \tiles{R_1}\sqcup\cdots\sqcup\tiles{R_k}\cup\{*\},
  \overline{H_{R_1}\sqcup\cdots\sqcup H_{R_k}} \right), 1\leq i\leq h_S(X).$$
The set of states are the set of tiles of the representative sets
of depth $i$. These tile sets may overlap, thus we need to take the
disjoint union. This causes no confusion since for each positioned
chain we know the current state of approximation, hence we know which
set of tiles we need to choose from.

The transformations come from the holonomy groups of  the representatives
of depth $i$.
How does $H_i$ act on $\cT_i$? If $P$ lies in the $j$th set $\tiles{R_j}$
of the disjoint union then $(h_1,\ldots,h_k)\in H_i$ acts on $P$ by
applying $h_j$ and it acts on $*$ trivially. Recall that
$\overline{H_i}$ augments the group $H_i$ with all constant maps on $\cT_i$. 
Since $(\cT_i,\overline{\holonomy{i}})$ is a well defined
transformation semigroup
for $1\leq i \leq h_S(X)$, we can form their wreath product. 

\begin{definition}
We call $\cH_1\wr\cdots\wr\cH_d=\cH(X,S)$  the \emph{holonomy wreath product semigroup} of $(X,S)$.
\end{definition}

In practice, we only want a substructure of this potentially huge
wreath product, so we need to construct a cascade product by giving
explicit dependency functions in the transformation cascades induced
by the generators of $S$. The maps are $s\mapsto \hat{s} \mapsto
\enc(\hat{s})$, where the final encoding describes $\hat{s}$ in terms
of the corresponding representative set.

\subsubsection{Encoding and decoding}
We encode the elements of a positioned chain, that are tiles of the
current state of approximation,  as tiles of the representative set of
the corresponding height.
If $\bC^{\pos}[j]\neq *$ then
$$\enc(\bC^{\pos})[j]=\bC^{\pos}[j]\cdot\torep{P} \text{ where }
P=\alpha_j(\bC^{\pos}),$$
otherwise the encoded value is *.
Since $\alpha$ is not recursive, encoding can also be done
independently for any level.

Decoding does the opposite, however we need to calculate the current
unencoded state of approximation, therefore it is a recursive calculation. Let
$\bV=\enc(\bC^{\pos})$, the tuple of coordinate values.
If $\bV^{\pos}[j]\neq *$ then
$$\dec(\bV^{\pos})[j]=\bV^{\pos}[j]\cdot\fromrep{P} \text{ where }
P=\alpha_j(\dec(\bV^{\pos})),$$
otherwise the encoded value is *.
These are bijective maps, thus $\dec(\enc(\bC^{\pos}))=\bC^{\pos}$ and
$\enc(\dec(\bV))=\bV$.
 
\subsubsection{Dependency functions}
For $\hat{s}$ in a chain semigroup $\hat{S}$, let's define $\enc(\hat{s})$ to be the transformation cascade given by
the dependency functions
$$\left( \enc(\hat{s}) \right)_i: \cT_1\times\cdots\times \cT_{i-1}\rightarrow 
\overline{H_i}.$$
Let's fix a positioned chain $\bC^{\pos}$, and thus
$P=\alpha_i(\bC^{\pos})$, $Q=\alpha_i(\hat{s}(\bC^{\pos}))$ and $\bV=\enc(\bC^{\pos})$.
By Lemma~\ref{chainaction},  these state approximations satisfy 
$P\cdot s \subseteq Q$. 

We need to define the value of the dependency function $\left(
  \enc(\hat{s}) \right)_i$ on $(V_1,\ldots,V_{i-1})$, the prefix of
$V$:
It is  constant $* \in \overline{\cH_i}$ unless we are on the right level, i.e.~$i=\depth{Q}$, in which case we have a constant
map (reset) to a tile  or a permutation. 

Precisely, 
if  $i\neq \depth{Q}$, then let $\enc(\hat{s})_i(V_1,\ldots,V_{i-1})=\mbox{ constant }* \in \overline{\cH}_i$. 

There there are two possibilities when $i=\depth{Q}$: 

\begin{description}
\item[Permutation] If the chain action satisfies $P\cdot s=Q$. The encoding of $s$ at depth $i$ on chains that agree with
$\bC^{\pos}$ up to depth $i-1$ is
  $\fromrep{P}\,s \,\torep{Q}$, a permutation of
  $\overline{P}=\overline{Q}$, therefore
$$\enc(\hat{s})_i(V_1,\ldots,V_{i-1})=\fromrep{P}\, s
\,\torep{Q},$$ and this is in the component of the holonomy group of
$\overline{P}$ by the definition of holonomy groups, and has
identities elsewhere according to the disjoint union action.

\item[Reset] If the chain action satisfies $P\cdot s\subset Q$ according to
  $\hat{s}$ we take the tile $\enc(\hat{s}(\bC^{\pos}))[i]$ of the representative $\overline{Q}$ and let
$$\enc(\hat{s})_i(V_1,\ldots,V_{i-1})=\constant(\enc(\hat{s}(\bC^{\pos}))[i]).$$
Since $\hat{s}$ is consistent with chain structure this constant is the same for all
chains that agree with $\bC^{\pos}$ up to depth $i-1$.
Again, the value of the dependency function is in $\overline{H_i}$ by
the definition of the holonomy permutation-reset transformation semigroups.
\end{description}

It is clear that $\enc(\hat{s})_i$ is well-defined since any $\hat{s}$ in the chain semigroup is 
consistent with chain structure, and $\enc$ and $\dec$ are defined
level-by-level on chains (same prefix gives same result). 
Therefore we have an element $\enc(\hat{s})$ of the wreath product, i.e., $\enc(\hat{s})\in\cH(X,S)$

\begin{theorem}\label{embeddingThm}
$(\cC,\hat{S}) \cong
(\positionedchains,\hat{S}) \hookrightarrow \cH_1\wr\cdots\wr\cH_d=\cH(X,S)$,
where $d=h_S(X)$. 
\end{theorem}
\noindent
The image of such an embedding is called a {\em holonomy (decompostion) cascade product}. 
\begin{proof}  The isomorphism was shown in Fact~\ref{chainSemiIso}.
We show $\enc$ is an embedding of transformation semigroups from $(\cC^{\pos},\hat{S})$ to 
the wreath product. 
For the states, $\enc(\cC^{\pos})\subseteq \cT_1\times \cdots
\times\cT_d$ holds trivially.
We need to show that if  $\bV=\enc(\bC^{\pos})\in\enc(\cC^{\pos})$ then $\enc(\hat{s})(\bV)=\enc(\hat{s}(\dec(\bV)))$.

By looking at the $i$th position for each $1\leq i \leq d$, 
if $i\neq \depth{\hat{s}(\alpha_i(\bC^{pos}))}$ then
$\hat{s}(\bC^{pos})$ cannot have a tile in position $i$, it follows that
$* = \enc(\hat{s}(\bC^{pos}))_i$, which is equal to 
$$V_i \cdot \enc(\hat{s})_i(V_1,\ldots,V_{i-1})=V_i \cdot \constant *,$$ 
as required.

Otherwise, $i=\depth{\hat{s}(\alpha_i(\cC^{\pos}))}$, and we have two cases. If 
$\enc(\hat{s})_i(V_1,\cdots, V_{i-1})$ is a constant map to a tile, then the definition of $\enc(\hat{s})$ 
yields
$$V_i\cdot \enc(\hat{s})_i(V_1,\ldots,V_{i-1})=V_i\cdot \constant(\enc(\hat{s}(\bC^{\pos}))[i])=\enc(\hat{s}(\bC^{\pos}))[i],$$
as required.
Otherwise, the component action is a permutation, and then
\begin{align*}
\enc(\hat{s})(\bV)[i]&=V_i\cdot \enc(\hat{s})_i(V_1,\ldots,V_{i-1})\\
&=\enc(\bC^{\pos})[i]\cdot \fromrep{P}\, s \,\torep{Q}\\
&=\bC^{\pos}[i]\cdot\torep{P}\, \fromrep{P}\, s \,\torep{Q}\\
&=\bC^{\pos}[i]\cdot s \,\torep{Q}\\
&=\enc(\hat{s}(\bC^{\pos}))[i]=\enc(\hat{s}(\dec(\bV)))[i]
\end{align*}
by the property that $\torep{P}\,\, \fromrep{P}=1_P$, the identity
map on $P$, hence on its set of tiles, where
$P=\alpha_i(\bC^{\pos})$ and $Q=\alpha_i(\hat{s}(\bC^{\pos}))$.\\
Since this holds for all $1\leq  i \leq d$,  we have
$$\enc(\hat{s})(\enc(\bC^{\pos}))=\enc(\hat{s}(\bC^{\pos})).$$ 
\noindent 
It follows that 
\begin{align*}
(\enc(\hat{s}_2)\circ \enc(\hat{s_1}))(\enc(\bC^{\pos}))& =
\enc(\hat{s}_2)(\enc(\hat{s_1})(\enc(\bC^{\pos}))\\
&=\enc(\hat{s}_2)(\enc(\hat{s}_1(\bC^{\pos})))\\
&=\enc(\hat{s}_2(\hat{s}_1(\bC^{\pos})))\\
&=\enc(\hat{s_2}\circ \hat{s_1})(\enc(\bC^{\pos})).
\end{align*}
Thus, $\enc$ is clearly an (injective) semigroup homomorphism. 
Whence, $(\enc(\cC^{\pos}),\enc(\hat{S}))$ is a isomorphic to $(\cC^{\pos},\hat{S})$. 
\end{proof}

By Lemma~\ref{liftLemma} and Theorem~\ref{embeddingThm}, we have
\begin{corollary}[Holonomy Decomposition Theorem]
A finite transformation semigroup $(X,S)$ divides its holonomy wreath product
$$(X,S) \mid \cH_1\wr\cdots\wr\cH_d=\cH(X,S),$$
where $d=h_S(X)$. 
\end{corollary}


\section{Computational Complexity}

The holonomy decomposition algorithm given here enumerates the image set $\imgs_S(X)$ of the state set $X$.
The worst case is enumerating the powerset with $2^{|X|}$ elements.
It is easy to conclude that the algorithm given has time complexity at least exponential in the number of states (cf.~Maler~\cite{Maler2010}).   
Moreover,  by the Krohn-Rhodes prime decomposition theorem \cite{primedecomp65,primedecomp68},  every simple group divisor of a finite semigroup must
occur as a divisor of any cascade decomposition. Therefore it follows that
a finite automata has no nontrivial subgroups (i.e., is {\em aperiodic}) if and only if all its holonomy groups are trivial. The
results of Cho and Huynh~\cite{aperiodicitypscpace91} show that aperiodicity is $PSPACE$-complete, so it follows immediately that computing the holonomy decomposition is $PSPACE$-hard. 

In practice we can calculate with huge semigroups (of size hundreds of thousands of elements). The size of the
state set and the size of generator set or of the semigroup do not necessarily give a good guide to computational complexity in practice. 
It would be interesting to find the appropriate features and parameters and do parametrized complexity analysis for holonomy decompositions. 

\section{Computed Examples}

\begin{example}
\label{ex:1}
\begin{figure}
\includegraphics[width=.4\textwidth]{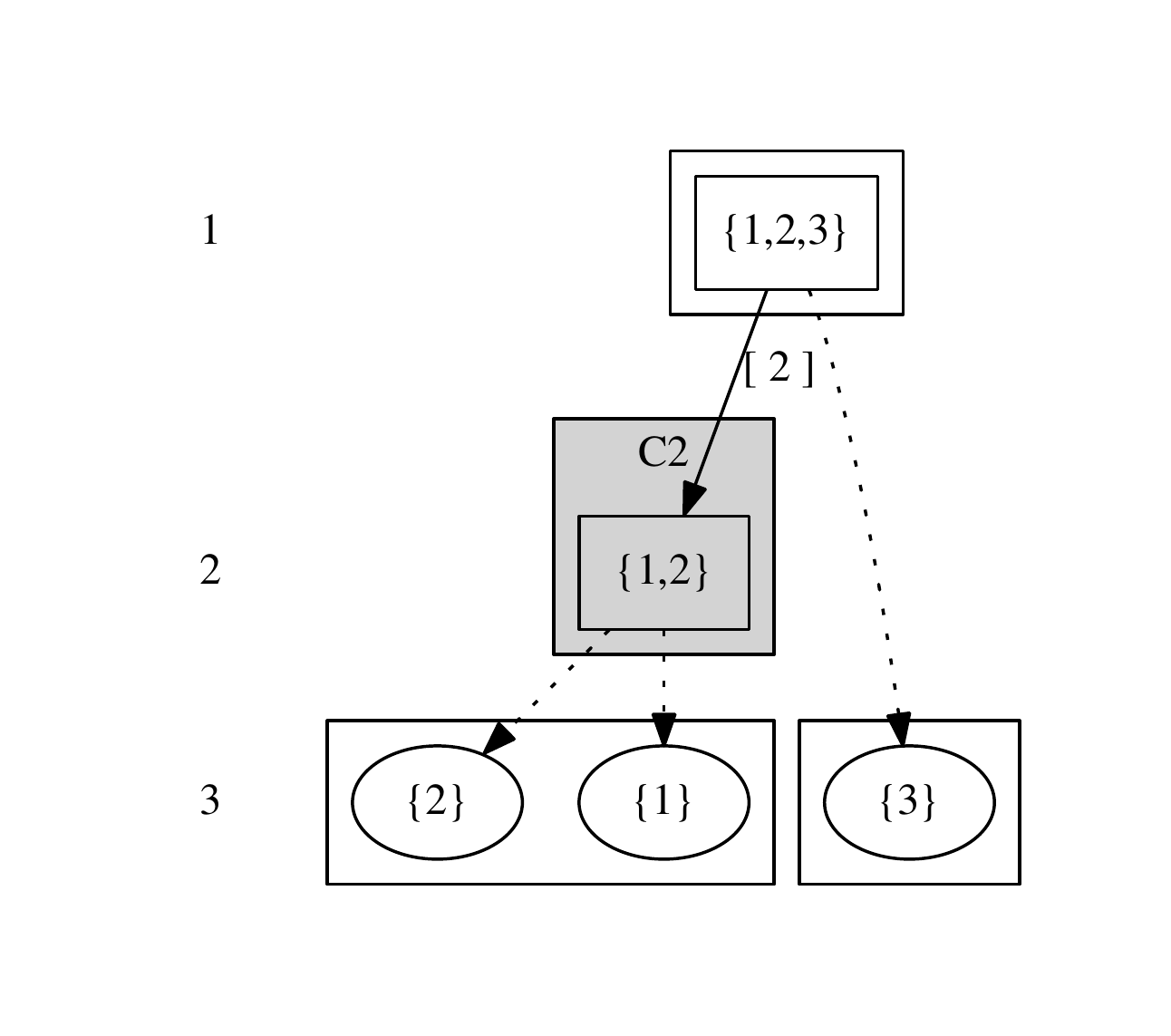}
\caption{The skeleton of a semigroup acting on 3 points (Example \ref{ex:1}). The nodes are the elements of $\extdimgs_S(X)$. The
  boxes are the equivalence classes, the rectangular nodes the chosen
  representatives of a class. Shaded equivalence classes have
  nontrivial holonomy groups. The arrows point to the tiles of
  a representative set, the labels denote sequences of generators
  taking the set to its tile. Dotted arrows indicate tiles that are not images. On the side depth values are indicated.}
\label{fig:small}
\end{figure} 

As a minimalistic but non-trivial example, let $({\bf 3},S)$ be the transformation semigroup generated by
$s_1=[2, 1 ,3]$ and $s_2= [ 1, 2, 2 ]$. From Figure \ref{fig:small} we
can read off the maximal chains: $\left\{ \{1,2,3\}, \{1,2\}, \{1\}
\right\}$, $\left\{ \{1,2,3\}, \{1,2\}, \{2\} \right\}$, $\left\{
      \{1,2,3\}, \{3\} \right\}$. 
Let's see how from $t=s_2s_1=[2 1 1]$ we construct $\hat{t}$ acting on the
chain representing state 1, i.e.~doing the action on the members of
the chain, removing duplicates then finding a dominating chain.
\begin{center}
\begin{tabular}{|c|c|c|c|}
\hline
chain $\bC$ & $\bC\cdot t$ & $\bC\cdot\hat{t}$\\
\hline
$\{1,2,3\}$ & $\{1,2\}$& $\{1,2,3\}$\\
\hline
$\{1,2\}$ &$\{1,2\}$ & $\{1,2\}$\\
\hline
$\{1\}$ & $\{2\}$ & $\{2\}$\\
\hline
\end{tabular}
\end{center}
In this very small example we have only a single dominating chain.
\end{example}

\begin{example}
\label{ex:2}
\begin{figure}
\includegraphics[width=.8\textwidth]{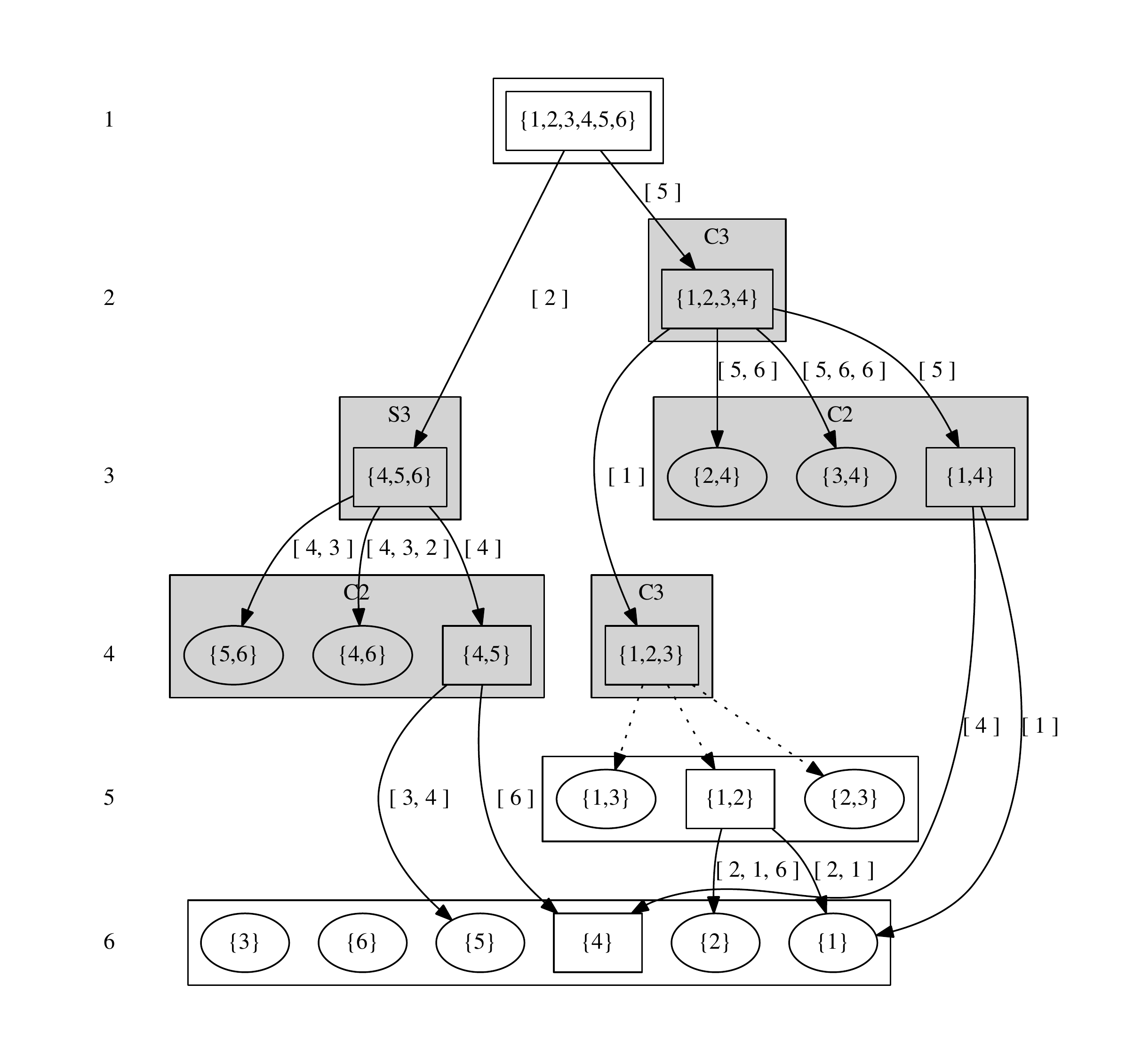}
\caption{The skeleton of a semigroup acting on 6 points (Example
  \ref{ex:2}). Interesting feature is that $\depth{\{1,4\}}<\depth{\{1,2,3\}}$.}
\label{fig:becks}
\end{figure} 

 Let $({\bf 6}, S)$ be the transformation semigroup
  generated by transformations $s_1,\ldots, s_6$:
\begin{center}
\begin{tabular}{ll}
$s_1=[1\ 2\ 3\ 1\ 1\ 1]$ & creates the image $\{1,2,3\}$,\\
$s_2=[4\ 4\ 4\ 5\ 4\ 6]$ & is the transposition $(4,5)$ and gives the
                           image $\{4,5,6\}$,\\
$s_3=[4\ 4\ 4\ 5\ 6\ 4]$ & is a cycle on $\{4,5,6\}$,\\
$s_4=[4\ 4\ 4\ 4\ 5\ 5]$ & creates the image $\{4,5\}$,\\
$s_5=[4\ 4\ 4\ 1\ 2\ 3]$ & maps $\{4,5,6\}$ to $\{1,2,3\}$,
                           and $\{1,2,3\}$ to $\{4\}$,\\
$s_6=[2\ 3\ 1\ 4\ 4\ 4]$ & is a cycle on  ${\{1,2,3\}}$,
\end{tabular}
\end{center}
and its basic properties are $|S|=138$, and $\extdimgs_S(X)=19$. The
`skeleton' of its holonomy decomposition is depicted on Figure \ref{fig:becks}.

For $\bC=\left\{ \{1,2,3,4,5,6\}, \{1,2,3,4\}, \{2,4\} ,\{2\}\right\}$
we have
\begin{center}
\begin{tabular}{|c|c|c|c|}
\hline
depth & $\bC^{\pos}$ & $\enc(\bC^{\pos})$ & $\alpha(\bC^{\pos})$\\
\hline
1 & $\{1,2,3,4\}$ &  $\{1,2,3,4\}$ & \{1,2,3,4,5,6\}\\
2 & $\{2,4\}$ & $\{2,4\}$ & \{1,2,3,4\}\\
3 & $\{2\}$& $\{1\}$ & \{2,4\}\\
4 & *& * & $\{2\}$\\
5 & *& * & $\{2\}$\\
\hline
\end{tabular}
\end{center}
\end{example}
demonstrating that an encoded positioned chain is not necessarily a
chain.

\begin{example} The full transformation semigroup $\cT_3$ has a
  canonical generating set consisting of two permutations
  (transposition and cycle) and an elementary collapsing. Figure
 ~\ref{fig:T3chainactions} shows how these generators act on the set
  of maximal chains. 
  The generators $[2,1,3]$ and $[2,3,1]$  are permutations of $\{1,2,3\}$ that map maximal chains to maximal chains. The lifts of
these transformations to the chain semigroup are thus exactly as shown the Figure~\ref{fig:T3chainactions} and hence unique. 
 The  transformation $[1,2,1]$ gives subsets chains $\left\{ \{1,2\},
    \{1\}\right\}$ and $\left\{ \{1,2\}, \{2\}\right\}$ that miss the full
state set itself. Thus, $[1,2,1]$ does not map maximal chains to maximal chains,  but it maps $\{1,2,3\}$ to $\{1,2\}$. For any lift  of $[1,2,1]$: All maximal chains trivially agree down to $\{1,2,3\}$  so it must map to a chains agreeing down to $\{1,2\}=\{1,2,3\}\cdot [1,2,1]$, thus the lift the of $[1,2,1]$ maps each maximal chain $\bC$ to
$\{1,2,3\} \supset \{1,2\} \supset \{\eta({\bC}) \cdot [1,2,1]\}$, and so is uniquely determined. 

However, having a unique dominating chain or unique lift is not a
general property. Constant map $c=[3,3,3]$ produces the chain $\left\{
  \{3\} \right\}$ for which any maximal chain containing  $\{3\}$ is a
dominating chain. 
Since any two maximal chains $\bC_1$ and $\bC_2$ both start with the top set $X=\{1,2,3\}$, they agree
at  $X$ and so, by consistency  $\hat{c}(\bC_1)$ and $\hat{c}(\bC_2)$ must agree down to $X \cdot c  = \{3\}$. 
That is, $\hat{c}(\bC_1)=\hat{c}({\bC_2})$, and  $\hat{c}$ is itself a constant map.
Here  there are two choices, $\{1,2,3\} \supset \{1,3\} \supset \{3\}$ or $\{1,2,3\} \supset \{2,3\} \supset \{3\}$,  for the constant value of $\hat{c}$.  

The same argument applies to lifting any constant map in this holonomy method: the lift of a constant to the chain semigroup yields a (non-unique) constant.

\begin{figure}
\begin{center}
\begin{tabular}{|c|c|}
\hline
$\maximalchains$& $\maximalchains\cdot [2,1,3]$ \\
\hline
\begin{tikzpicture}
[level distance=1cm,level 1/.style={sibling distance=1.3cm},level 2/.style={sibling distance=.6cm}]

\node {$\{1,2,3\}$} child { node{$\{1,2\}$} child { node{$\{1\}$}} child { node{$\{2\}$}}} child { node{$\{1,3\}$} child { node{$\{1\}$}} child { node{$\{3\}$}}} child { node{$\{2,3\}$} child { node{$\{2\}$}} child { node{$\{3\}$}}};\end{tikzpicture} &\begin{tikzpicture}
[level distance=1cm,level 1/.style={sibling distance=1.3cm},level 2/.style={sibling distance=.6cm}]

\node {$\{1,2,3\}$} child { node{$\{1,2\}$} child { node{$\{2\}$}} child { node{$\{1\}$}}} child { node{$\{2,3\}$} child { node{$\{2\}$}} child { node{$\{3\}$}}} child { node{$\{1,3\}$} child { node{$\{1\}$}} child { node{$\{3\}$}}};\end{tikzpicture}\\
\hline
\end{tabular}

\begin{tabular}{|c|c|}
\hline
$\maximalchains\cdot[2,3,1]$& $\maximalchains\cdot [1,2,1]$ \\
\hline
\begin{tikzpicture}
[level distance=1cm,level 1/.style={sibling distance=1.3cm},level 2/.style={sibling distance=.6cm}]

\node {$\{1,2,3\}$} child { node{$\{2,3\}$} child { node{$\{2\}$}} child { node{$\{3\}$}}} child { node{$\{1,2\}$} child { node{$\{2\}$}} child { node{$\{1\}$}}} child { node{$\{1,3\}$} child { node{$\{3\}$}} child { node{$\{1\}$}}};\end{tikzpicture} & \begin{tikzpicture}
[level distance=1cm,level 1/.style={sibling distance=1.3cm},level 2/.style={sibling distance=.6cm}]

\node {$\{1,2\}$} child { node{$\{1,2\}$} child { node{$\{1\}$}} child { node{$\{2\}$}}} child { node{$\{1\}$} child { node{$\{1\}$}} child { node{$\{1\}$}}} child { node{$\{1,2\}$} child { node{$\{2\}$}} child { node{$\{1\}$}}};\end{tikzpicture}\\
\hline
\end{tabular}
\end{center}
\caption{Action of the canonical generators (transposition, cycle,
  elementary collapsing) on the maximal chains.}
\label{fig:T3chainactions}
\end{figure}
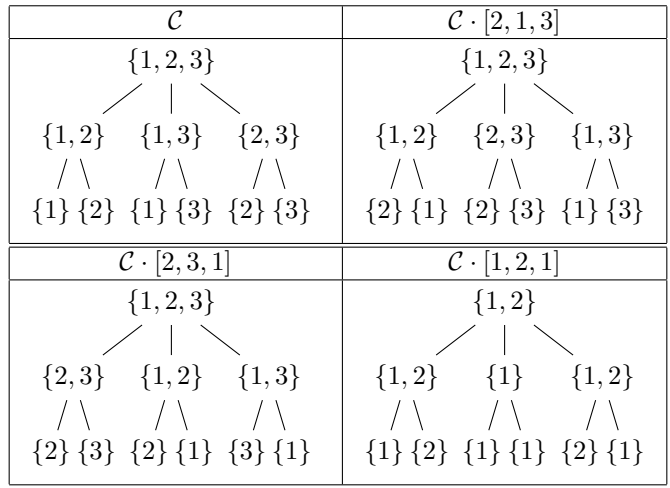

\end{example}

\vspace{1em}\noindent
{\bf Acknowledgments.}  The work of the authors was supported in part by the European Commission's Seventh Framework Programme Future and Emerging Technology (FET) project BIOMICS,
grant agreement contract no.\ 318202.  This support is gratefully acknowledged.

\bibliographystyle{plain}
\bibliography{coords}

\appendix
\section{Notation}

\begin{center}
\begin{tabular}{ll}
$\subseteq_S$ & subduction relation\\
$\prec$ & tile of relation \\
$\alpha_i(\bC^{\pos})$ & current state of approximation at depth $i$
                         for a positioned chain\\
$\bC$, $\bD$ & chains\\
$\bC^{\pos}$, $\bC^{\pos}[i]$ & positioned chains, content of position
                                $i$\\
$\maximalchains_P$, $\maximalchains_x$, $\maximalchains$ & maximal
                                                           chains from
                                                           $P$, $\{x\}$ to $X$, all maximal chains.\\
$(\maximalchains,\hat{S})$ & chain semigroup\\
$\cC^{\pos}$ & all positioned tile chains\\
$\permutator{P}$ &the permutator (generalized Sch\"utzenberger) group  of $P$\\
$\holonomy{P}$ & holonomy permutation group of  $P$, $(\tiles{P},\permutator{P})$ made faithful\\
$\overline{\holonomy{P}}$ & holonomy permutation-reset transformation semigroup of $P$\\
$h_S(P), \depth{P}$ & height, depth of a set\\
$\fromrep{P}$, $\torep{P}$ & mapping from and to a representative\\
$P,Q\in \imgs_{S}(X)$, $\imgs'_{S}(X)$ & images, image set, extended image set\\
$P\equiv_S Q,\rep{P}$ & equivalence, representative element \\

$S_P$ & setwise stabilizer semigroup of $P$\\
$\tiles{Q}$& the tiles of $Q$\\
$\bV$ & encoded coordinate values (tiles of representatives)\\
$(X,S)$, $(Y,T)$ & transformation semigroups\\
$x,y,z \in X$ & states, stat set
\end{tabular}
\end{center}

\end{document}